\documentclass{article}
\usepackage{amsmath} 
\usepackage{amsthm}
\usepackage{amssymb}
\usepackage{latexsym}
\usepackage[center]{caption}
\usepackage{tikz-cd}
\usepackage{enumitem}
\usepackage{verbatim}
\usepackage[backend=biber,bibencoding=utf8]{biblatex}

\newtheorem{theorem}{Theorem}[section]
\newtheorem{lemma}[theorem]{Lemma}

\newtheorem{corollary}[theorem]{Corollary}

\newtheorem{definition}[theorem]{Definition}

\def\Mod{\mathrm{Mod}}

\def\Aut{\mathrm{Aut}}

\def\Homeo{\mathrm{Homeo}}
\def\Ht{\mathrm{Ht}}

\title{Note on Mapping Class Groups of Finite Spaces}
\author{Ben Branman}
\begin{document}
\maketitle

\begin{abstract}
We investigate the mapping class groups of a class of non-Hausdorff topological spaces which includes finite spaces.  We show that the mapping class group of a finite space is isomorphic to the homeomorphism group of its $T_0$ quotient.  As a corollary, we show that every finite group is the mapping class group of some finite, $T_0$ space.
\end{abstract}
\section{Introduction}
The goal of this paper is to explore mapping class groups of finite topological spaces.  If $X$ is a topological space, the mapping class group $\Mod(X)$ is the group of isotopy classes of homeomorphisms of $X$.  If $X$ is a compact, orientable, noncontractible surface, then two homeomorphisms of $X$ are isotopic if and only if they are homotopic \cite{primer}.  However, for arbitrary spaces this is false.  The mapping class groups of manifolds have been well-studied, but little is known about mapping class groups of non-Hausdorff spaces, which includes all non-discrete finite spaces.
\par 
Our main results are as follows.  First, we show that it is sufficient to study the mapping class groups of $T_0$ spaces.  Recall that there is a functor $T0$ from the category of topological spaces to the category of $T_0$ spaces.  For a space $X$, $T0(X)$ is the quotient space obtained by identifying topologically indistinguishable points.  $T0$ is right-adjoint to the inclusion functor from $T_0$ spaces to all spaces.
\par 
Our main results are as follows.

\begin{theorem}
Let $X$ be an arbitrary topological space.  Then $\Mod(X)\cong \Mod(T_0(X))$
\end{theorem}

Thus, to understand mapping class groups of non-Hausdorff spaces, it  is sufficient to study mapping class groups of $T_0$ spaces.  We will show the following.  

\begin{theorem}
Let $X$ be a finite $T_0$ space and let $f,g\in \Homeo(X)$.  Then $f$ is isotopic to $g$ if and only if $f=g$.  In particular, $\Mod(X)=\Homeo(X)$.
\end{theorem}

We then use a result of Barmak\cite{Barmak2009} to show
\begin{corollary}
Let $G$ be a finite group and let $X$ be a finite space.  Then there exists a finite space $Y$ such that $X$ and $Y$ are (strong) homotopy-equivalent and such that $\Mod(Y)\cong G$.
\end{corollary}

Combining this result with a classical theorem of McCord\cite{McCord1966}, we obtain another corollary.

\begin{corollary}
Let $G$ be a finite group and let $X$ be a finite simplicial complex.  Then there exists a finite space $Y$ which is weak equivalent to $X$ such that $\Mod(Y)\cong G$.
\end{corollary}

  A finite topological space is $T_0$ if and only if it is the order topology of a finite partially ordered set.  Our proof of Theorem 1.2 actually works for a wider class of posets which includes finite posets.  By suitably weakening the hypothesis of Theorem 1.2, we obtain the following corollary.

\begin{corollary}
There exists a countably infinite, contractible poset $X$ such that $\Mod(X)=\Homeo(X)\cong S_{\mathbb{N}}$.
\end{corollary}

\section{Spaces which are not $T_0$}
The goal of this section is to prove Theorem 1.1.  First, we will need the following lemma.

\begin{lemma}
Let $X$ be an arbitrary topological space, and let $U\subseteq X$ be an indiscrete subset.  Let $f:X\to X$ be a (not necessarily continuous) bijection such that $f_{|X\backslash U}$ is the identity.  Then
\begin{enumerate}
\item
The map $f$ is a homeomorphism.
\item
The homeomorphisms $f$ and $\mathrm{Id}_X$ are isotopic.  Moreover, they are isotopic relative to $X\backslash U$.
\end{enumerate}
\end{lemma}
\begin{proof}
First we must show that $f$ is continuous.  Let $V\subseteq X$ be an open set.  If $U\cap V=\emptyset$, then $f^{-1}(V)=V$ is also open.  On the other hand, if $U\cap V$ is nonempty, then $U\subset V$ because $U$ is a set of topologically indistinguishable points.  Since $f^{-1}(U)=U$, then $f^{-1}(V)=V$.  Hence, $f$ is continuous.  A symmetric argument shows that $f^{-1}$ is continuous, hence $f$ is a homeomorphism.  
\par 
Now we define a map $H: X\times I\to X$ given by
\[
H(x,t)=\begin{cases}
x & 0\leq t<1 \\
f(x) & t=1
\end{cases}
\]

We claim that $H$ is an isotopy.  It is clear from the previous paragraph that $h_t=H(*,t)$ is a homeomorphism for all $t\in I$.  Thus, we only need to show that $H$ is continuous.  For an open set $V\subseteq X$, the argument in the previous paragraph shows that $f^{-1}(V)=V$.  Hence, $H^{-1}(V)=V\times I$, which is open in $X\times I$, so $H$ is continuous.
\end{proof}
Now we are ready to prove Theorem 1.1.  
\begin{proof}
Suppose that $f,g\in \Homeo(X)$ such that $f(x)$ and $g(x)$ are topologically indistinguishable for all $x\in X$.  We want to show that $f$ and $g$ are isotopic.  By composing each homeomorphism with $g^{-1}$, we may assume that $g=\mathrm{Id}_X$.  Hence, $x$ and $f(x)$ are topologically indistinguishable for all $x\in X$.  Applying Lemma 2.1 to each indiscrete subset of $X$ yields and isotopy between $f$ and the identity.
\end{proof}

\section{Mapping Class Groups of Posets}
Let $P$ be a partially ordered set.  Recall that a \emph{lower set} of $P$ is a set $S\subseteq P$ such that, whenever $a\leq b$ and $b\in S$, then $a\in S$.  Then $P$ can be given a topology, called the \emph{order topology}, by declaring the lower sets to be open.  Conversely, it is well known (see, for example \cite{BarmakFiniteSpacesBook}) that a space $X$ is homeomorphic to a poset with the order topology if and only if $X$ is $T_0$ and if arbitrary intersections of open sets in $X$ are open.  A map between posets is order preserving if and only if it is continuous in the order topology. 
\par 
Any finite $T_0$ space clearly satisfies the above conditions, and this fact gives an equivalence of categories between finite $T_0$ spaces and finite partially ordered sets.  Thus, Theorem 1.2 is really a statement about finite posets.
\par 
Before proving Theorem 1.2, we need to define two concepts.

\begin{definition}\rm
Let $P$ be a poset.  A \emph{chain of length $n$} is a sequence of distinct points $a_0<a_1<\cdots <a_{n-1}<a_n$ in $P$.   The \emph{height} of an element $a\in P$ is the supremum of the lengths of chains whose maximum are $a$:
\[\Ht(a)=\sup\{n\in \mathbb{N}\cup \{\infty\}: \exists a_0<\cdots <a_n=a\}\]
\end{definition}
Note that if $a<b$ and $b$ has finite height, then $\Ht(a)<\Ht(b)$. Also note that if $f: P\to P'$ is a homeomorphism, then $\Ht(f(a))=\Ht(a)$ for all $a\in P$.

\begin{definition}\rm
Let $P$ be a poset, and let $a,a'\in P$.  Then $a$ and $a'$ are said to be \emph{down-equivalent} if for all $b\in P$, $b<a$ if and only if $b<a'$.  
\end{definition}
Note that down-equivalence is an equivalence relation.  Also note that if $f: P\to P'$ is a homeomorphism, then $a$ is down-equivalent to $a'$ if and only if $f(a)$ is down-equivalent to $f(a')$.  
\par 
We are now ready to state the main result of this section.

\begin{theorem}
Let $P$ be a poset such that
\begin{itemize}
\item
Each element of $P$ has finite height,
\item
Each down-equivalence class of $P$ is countable.
\end{itemize}
Then any two elements of $\Homeo(P)$ are isotopic if and only if they are equal.
\end{theorem}
A finit poset clearly satisfies both conditions, and so Theorem 1.2 follows as an immediate corollary of Theorem 3.3.

\begin{proof}
Suppose $f,g\in \Homeo(P)$ and let $H:P\times I\to P$ be an isotopy between $f$ and $g$.  By composing both homeomorphisms with $g^{-1}$, we obtain an isotopy between $f\circ g^{-1}$ and the identity, thus we may assume without loss of generality that $g=\mathrm{Id}_P$.
\par 
For $0\leq t\leq 1$, let $H_t\in \Homeo(P)$ be the map given by $H_t(p)=H(p,t)$.  We will prove that $H_t$ fixes every element of $P$ by using strong induction on height, and hence in particular that $H_0=f$ fixes every element of $P$.  We can use induction because each point has finite height.  First, consider the subspace $P_0\subseteq P$ of points of height zero.  It is evident from the definitions that $P_0$ carries the discrete topology.  Since homeomorphisms preserve height, any element of $\Homeo(P)$ restricts to an element of $\Homeo(P_0)$.  Thus, the restriction of $H$ to $P_0\times I$ gives an isotopy between $f\vert_{P_0}$ and $\mathrm{Id}_{P_0}$.  But $P_0$ is discrete, so the only homeomorphism of $P_0$ isotopic to the identity is the identity itself.  Thus $H_t\vert_{P_0}=\mathrm{Id}_{P_0}$.
\par 
Now suppose that for all $t\in I$ and for all $p\in P$ such that $\Ht(p)<n$, we have $H_t(p)=p$. Consider a point $a$ of height $n$.  For $0\leq t\leq 1$, $H_t$ is a homeomorphism of $P$.  Thus, for any $b\in P$, we have $b<a$ if and only if $H_t(b)<H_t(a)$.  But by the induction hypothesis, $H_t(b)=b$ whenever $b<a$.  Thus $b<a$ implies $b<H_t(a)$.  Moreover, $H_t^{-1}$ is also a homeomorphism of $P$ which is isotopic to the identity, so $b<H_t(a)$ implies $b<a$.  In other words, $a$ and $H_t(a)$ are down-equivalent.
\par 
Now let $A$ be the down-equivalence class of $a$.  For $b\in A$, define 
\[X_b=\{t\in I: H(a,t)=b\}\]
We will show that $X_b$ is closed in $\mathbb{R}$.  Suppose $t_1,t_2,\ldots$ is a Cauchy sequence in $X_b$ and that $\lim_{m\to \infty} t_m=t$.  We want to show that $t\in X_b$.  Since $H$ is continuous, and continuity implies sequential continuity, $H$ is also sequentially continuous.  Hence, $H(a,t)$ must be a limit of the sequence $H(a,t_1),H(a,t_2),\ldots$.  But this sequence is the constant sequence $b,b,b,b,\ldots$.  However, we are not quite ready to conclude that $H(a,t)=b$, because $P$ is non-Hausdorff and so limits are not unique.  
\par 
It is immediate from the definition of convergence that a point $c\in P$ is a limit of the constant sequence $\{b\}_{m\geq 1}$ if and only if $c\leq b$.  Thus, $H(a,t)\leq b$.  If $H(a,t)<b$, then by the definition of down-equivalence we would have $H(a,t)<a$, which is impossible since $H_t$ is an order isomorphism.  Thus $H(a,t)=b$, so $t\in X_b$ and hence $X_b$ is closed.
\par 
By construction, $I$ can be written as the disjoint union
\[I=\bigsqcup_{b\in A} X_b\]
Moreoever, by the previous two paragraphs, each $X_b$ is closed.  Also recall that $A$ is countable by hypothesis.  Sierpinski's Theorem states that if $I$ can be written as a disjoint, countable union of closed sets, then exactly one of the closed sets is nonempty.  Thus, $H(a,t)$ does not depend on $t$.  Since $H(a,0)=a$ by hypothesis, we can conclude that $H(a,t)=a$ for all $t\in I$.   
\end{proof}

Now we are ready to prove Corollary 1.5.

\begin{proof}
Let $X$ be the set $\mathbb{N}\cup \{x\}$ with the partial order $a\leq b$ if and only if $a=x$.  Then the point $x$ has height 0, and all other points in $X$ have height 1.  $X$ is also countable.  Thus, this poset satisfies the hypotheses of Theorem 3.3, and thus $\Aut(X)\cong \Mod(X)$.  It is also clear from the construction that $\Aut(X)$ is isomorphic to the symmetric group on a countably infinite set.
\end{proof}
\printbibliography
\end{document}